\documentclass[12pt, reqno]{amsart}

\usepackage{amsmath}
\usepackage{amssymb,amsfonts,amscd}
\usepackage{latexsym}
\usepackage{mathrsfs}
\usepackage{color}

\makeatletter
\def\@settitle{\begin{center}%
    \baselineskip14\p@\relax
    \bfseries
    \@title
  \end{center}%
}
\makeatother

\usepackage{hyperref}

\newcommand\al{\alpha}
\newcommand\be{\beta}
\newcommand\ga{\gamma}
\newcommand\Ga{\Gamma}
\newcommand\de{\delta}
\newcommand\De{\Delta}

\newcommand\Om{\Omega}

\newcommand\la{\lambda}

\newcommand\si{\sigma}

\newcommand\R{\mathbb R}
\newcommand\C{\mathbb C}

\newcommand\Z{\mathbb Z}
\newcommand\Y{\mathbb Y}
\newcommand\E{\mathbb E}

\newcommand\A{\mathscr A}

\newcommand\const{\operatorname{const}}

\newcommand\Sym{\operatorname{Sym}}
\newcommand\Tab{\operatorname{RTab}}
\newcommand\fin{{\operatorname{fin}}}

\newcommand\alde{{\al,\be,\ga,\de}}

\newcommand\wt{\widetilde}
\newcommand\wh{\widehat}

\newcommand\ccdot{\,\cdot\,}

\newtheorem{theorem}{Theorem}[section]
\newtheorem{proposition}[theorem]{Proposition}

\newtheorem{corollary}[theorem]{Corollary}

\theoremstyle{definition}
\newtheorem{definition}[theorem]{Definition}
\newtheorem{remark}[theorem]{Remark}
\newtheorem{example}[theorem]{Example}

\numberwithin{equation}{section}

\begin{document}

\title[]{\large An analogue of big $q$-Jacobi polynomials in the algebra of symmetric functions}

\author{Grigori Olshanski}
\address{Institute for Information Transmission Problems of the Russian Academy of  Sciences, Moscow, Russia; \newline \indent Skolkovo Institute of Science and Technology (Skoltech), Moscow, Russia};\email{olsh2007@gmail.com}


\thanks{The present research was carried out at the Institute for Information Transmission Problems of the Russian Academy of Sciences at the expense of the Russian Science Foundation (project 14-50-00150).
}

\begin{abstract}
It is well known how to construct a system of symmetric orthogonal polynomials with an arbitrary finite number of variables out of an arbitrary system of orthogonal polynomials on the real line. In the special case of the big $q$-Jacobi polynomials, the number of variables can be made infinite. As a result, in the algebra of symmetric functions, there arises an inhomogeneous basis whose elements are orthogonal with respect to some probability measure. This measure is defined on a certain space of infinite point configurations and hence determines a random point process.
\end{abstract}

\dedicatory{To the memory of Mikhail Semenovich Agranovich}

\maketitle

\section{Introduction}

Let $\Sym$ denote the algebra of symmetric functions over $\R$. The main result of the paper is a construction of a five-parameter family of new bases in $\Sym$. These bases are inhomogeneous and share many properties of systems of orthogonal polynomials on a segment. This means, in particular, that $\Sym$ is realized as a dense subalgebra of $C(\Om)$, the algebra of continuous functions on a certain compact space $\Om$, and under this realization, our bases turn into orthogonal bases of  Hilbert $L^2$ spaces on $\Om$ with respect to certain probability measures. Those measures on $\Om$  are of independent interest ---  they are an infinite-dimensional analogue of multidimensional $q$-Beta distributions. Important elements of our construction are: 1) the big $q$-Jacobi polynomials; 2) a generalization of  the Knop-Okounkov-Sahi multivariate interpolation polynomials to the case of infinite number of variables.   

\subsection{Orthogonal systems in an abstract algebra} 

We begin with a little formalism. Assume we are given a commutative unital graded algebra $\A=\bigoplus_{n=0}^\infty\A_n$ over $\R$,  such that  $\dim\A_n<\infty$ for all $n$ and $\A_0=\R1$. Further, assume that $\{f_\la\}$ is a fixed homogeneous basis in  $\A$ containing the unity $1$ (here $\la$ ranges over a set of indices). 

\begin{definition}\label{def1.A}
By an  \emph{orthogonal system in $\A$ adapted to a distinguished homogeneous basis $\{f_\la\}$} we mean the following data: 
 
 (i) $\A\to C(\Om)$ is an injective homomorphism of the algebra $\A$ into the algebra of continuous functions on a locally compact space $\Om$ with pointwise operations;
 
 (ii) $M$ is a probability Borel measure on $\Om$ such that all functions coming from elements of the algebra  $\A$ are square integrable with respect to $M$; 
 
 (iii)  $\{\pi_\la\}$ is an inhomogeneous basis in  $\A$ such that  
 $$
 \pi_\la=f_\la+\text{lower degree terms} 
 $$
and
 $$
 \langle \pi_\la, \pi_\mu\rangle=\de_{\la\mu}h_\la, \qquad h_\la>0, 
 $$
 for all indices $\la$ and $\mu$, where the angular brackets denote the inner product in $\A$ induced by the map $\A\to L^2(\Om,M)$.
 
 \end{definition}
 
In the case when $\A=\R[x]$ (the algebra of polynomials in one variable) with the distinguished basis $1,x,x^2,x^3,\dots$ we obtain the standard definition of a system of orthogonal polynomials on the real line.  

\begin{remark}
We attach to  $\{\pi_\la\}$  a linear functional $\E:\A\to\R$ (an analogue of the moment functional). Namely, given $f\in\A$, we expand it on the basis $\{\pi_\la\}$ and define $\E(f)$ as the coefficient of the unity element in this expansion. Then the inner product in $\A$ takes the form
$$
\langle f,g\rangle=\E(fg), \qquad f,g\in\A,
$$
and hence depends on the basis $\{\pi_\la\}$ only. It also follows that the basis $\{\pi_\la\}$ has to be inhomogeneous, because otherwise the inner product will degenerate.
\end{remark}

\begin{remark}
In the case $\A=\R[x]$, the orthogonal polynomials can be obtained via the Gram--Schmidt orthogonalization process applied to the monomials. In the abstract situation, when the homogeneous components $\A_n$ are not necessarily one-dimensional, this method is not directly applicable. Therefore, the very existence of an orthogonal system adapted to a given homogeneous basis should be viewed as a special effect. 
\end{remark}

\subsection{Example: symmetric orthogonal polynomials in $N$ variables}\label{sect1.A}
Denote by $\Sym(N)$ the algebra of symmetric polynomials in $N$ variables $x_1,\dots, x_N$. Set $\A=\Sym(N)$ and take as the distinguished homogeneous basis the Schur polynomials. They are indexed by the partitions $\la$ of length $\ell(\la)\le N$. 

Then there exists a well-known procedure of constructing an orthogonal system in $\Sym(N)$ starting from an arbitrary system $\{\pi_n=x^n+\dots\}$ of orthogonal polynomials on $\R$. Namely, the inhomogeneous basis consists of the polynomials 
$$
\pi_\la(x_1,\dots,x_N)=\frac{\det[\pi_{\la_i+N-i}(x_j)]_{i,j=1}^N}{\prod\limits_{1\le i<j\le N}(x_i-x_j)}.
$$
As the space $\Om$, one takes the subset of vectors in $\R^N$ with coordinates $x_1\ge\dots\ge x_N$, and the measure $M$ on $\Om$ ensuring the orthogonality property of the polynomials $\pi_\la$ has the form 
$$
\const \cdot\prod\limits_{1\le i<j\le N}(x_i-x_j)^2\cdot W(dx_1)\dots W(dx_N)\Big|_\Om,
$$
where $W(dx)$ is the weight measure of the initial system $\{\pi_n\}$.
 
\subsection{Results}

Recall that  $\Sym$ denotes the algebra of symmetric functions and take as the distinguished homogeneous basis the Schur symmetric functions $S_\la$. They are indexed by arbitrary partitions  $\la$. The main result of the paper is Theorem \ref{thm4.A}. Its short formulation, based on Definition  \ref{def1.A}, is the following.
\medskip

\noindent \textbf{Theorem.} \emph{
In the algebra \/ $\Sym$ there exists a $5$-parameter family of orthogonal systems adapted to the homogeneous basis $\{S_\la\}$}:
$$
\{\Phi_\la(\,\cdot\,; q,\alde)=S_\la(\,\cdot\,)+\text{\rm lower degree terms}\}.
$$ 

\medskip

Here $\la$ ranges over the set of partitions; the point in the parentheses denotes the collection of arguments of symmetric functions; the parameter $q$ lies in  $(0,1)$;  the parameters $\al$ and $\be$ are real, $\al>0$, $\be<0$; the parameters $\ga$ and $\de$ are also subject to some constraints: for instance, a sufficient condition is $\ga=\bar\de\in\C\setminus\R$. 

The Hilbert spaces $L^2(\Om,M)$ in which the algebra $\Sym$ is embedded look as follows. The space $\Om$ is compact and its elements are arbitrary point configurations (=subsets) in the countable set 
$$
\De_{q,\al,\be}:= \{\be^{-1}q,\be^{-1}q^2,\be^{-1}q^3,\dots\}\cup\{\dots, \al^{-1}q^3,\al^{-1}q^2,\al^{-1}q\}\subset\R\setminus\{0\}
$$
(a segment of a two-sided $q$-lattice). The probability measures $M=M^{q,\alde}$ are a particular case (the degenerate series) of the determinantal measures introduced in  \cite{GO-JFA}. Each of these measures is concentrated on the subset  $\Om_\infty\subset\Om$ of infinite configurations. 

An important role in the construction of the inhomogeneous symmetric functions $\Phi_\la(\,\cdot\,; q,\alde)$ is played by one more inhomogeneous basis in  $\Sym$ ---  the \emph{$q$-interpolation symmetric functions} $I_\mu(\,\cdot\,;q)$.  They are an analogue of the $q$-interpolation symmetric polynomials of Knop--Okounkov--Sahi. The symmetric functions $I_\mu(\,\cdot\,;q)$ can be defined axiomatically, for them there exists an analogue of Okounkov's combinatorial formula, and their expansion on the Schur functions $S_\nu$ can be written explicitly (Theorems \ref{thm2.A} and \ref{thm2.B}).

On the other hand, the functions $\Phi_\la(\,\cdot\,; q,\alde)$ can be expanded, in an explicit way, on the basis  $I_\mu(\,\cdot\,;q)$. From this we derive a two-step expansion of the functions $\Phi_\la(\,\cdot\,; q,\alde)$ on the Schur functions $S_\nu$ (Theorem \ref{thm3.A}). 

The measures $M^{q,\alde}$ and the interpolation symmetric functions $I_\mu(\,\cdot\,;q)$ seem to be of independent interest. 

\subsection{Connection with big $q$-Jacobi polynomials}

The classical \emph{big $q$-Jacobi polynomials} are orthogonal on the lattice $\De_{q,\al,\be}$ with a certain weight and can be expressed through the $q$-hypergeometric series ${}_3\phi_2$ (Andrews--Askey's \cite{AA}, Koekoek--Swarttouw \cite{KS}). In the hierarchy of $q$-orthogonal polynomials, they are next to the Askey--Wilson polynomials, which are at the top level and expressed through ${}_4\phi_3$.

The procedure described in \S\ref{sect1.A} allows one to introduce the $N$-variate symmetric big  $q$-Jacobi polynomials, and the latter are exploited in the construction of the symmetric functions  $\Phi_\la(\,\cdot\,; q,\alde)$. Namely, $\Phi_\la(\,\cdot\,; q,\alde)$ is obtained from the $N$-variate polynomials with the same index $\la$ by a limit transition as $N\to\infty$. Note that in this limit transition, some parameters of the $N$-variate polynomials vary together with $N$.

\subsection{Other examples of orthogonal systems in $\Sym$}

A ``lifting'' of the Koornwinder polynomials (which are a multivariate extension of the Askey--Wilson polynomials) to the algebra $\Sym$ is described by Rains  \cite[\S7]{Rains}. However, on the Askey--Wilson--Koornwinder level, the picture looks more complicated; in particular, orthogonality in $\Sym$ is understood in a more formal sense (existence of a moment functional). 

There are also examples of orthogonal systems in $\Sym$ on a much lower level of the hierarchy of hypergeometric polynomials:  these are the Meixner and Laguerre symmetric functions (author's papers \cite{Ols-JMS} and \cite{Ols-IMRN}, Petrov \cite[\S9.4.5]{P}). 

\subsection{Generalization}

All the results of the present paper admit an extension to the situation when the distinguished homogeneous basis consists of the Macdonald symmetric functions with two parameters $q$ and $t$ (the Schur functions are a particular case corresponding to $t=q$). However, I decided to separately examine the case $t=q$ first,  because this can be achieved by elementary tools, by using explicit determinantal formulas (in the spirit of \cite{OO-AA}). In the general case, when $t\ne q$, determinantal formulas are absent and the proofs become substantially more complicated. 

\subsection{Acknowledgment}
I am grateful to Vadim Gorin, Leonid Petrov, and especially to Cesar Cuenca for valuable comments. 

\section{Interpolation symmetric functions}

\subsection{Notation}
The symbols $\la,\mu,\nu$ denote partitions, which we identify with the corresponding Young diagrams. We use a standard notation:  

$\bullet$ length of partition  $\ell(\la):=\min\{i: \la_{i+1}=0\}$;

$\bullet$ size of partition  (=number of boxes in Young diagram) $|\la|:=\sum\la_i$; 

$\bullet$ transposing Young diagram $\la\mapsto\la'$; 

$\bullet$ $n(\la):=\sum (i-1)\la_i$;

$\bullet$  $\la\supseteq\mu$ or $\mu\subseteq\la$ means that diagram $\mu$ is contained in diagram  $\la$.

Recall the notation $\Sym$ and $\Sym(N)$ for the algebra of symmetric functions over 
 $\R$ and the algebra of symmetric polynomials in $N$ variables over $\R$, respectively. Here and in what follows the parameter $N$ takes the values $1,2,3,\dots$\,. The Schur function with index $\la$ is denoted by  $S_\la$ and the Schur polynomial with the same index and $N$ variables is denoted by $S_{\la\mid N}$; a standard agreement is that $S_{\la\mid N}=0$ if $\ell(\la)>N$. 

Next, denote by $\ell^1$ the real Banach space whose elements are vectors  $X=(x_1,x_2,\dots)$ with the norm $\Vert X\Vert=\sum|x_i|$. It is important for us that the elements $F\in\Sym$ can be interpreted as continuous functions $F(X)$ on $\ell^1$. 

Throughout the whole paper $q\in(0,1)$ is a fixed parameter. To every partition $\la$ we attach a vector 
$$
X(\la):=(q^{-\la_1}, q^{-\la_2+1}, q^{-\la_3+2},\dots)\in\ell^1.
$$

Notation of $q$-Pochhammer symbols: if $z$ is an arbitrary number, then 
$$
(z;q)_\infty:=\prod_{i=1}^\infty(1-zq^{i-1}), \quad (z;q)_n:=\prod_{i=1}^n(1-zq^{i-1})=\frac{(z;q)_\infty}{(zq^n;q)_\infty}.
$$
More generally, 
\begin{equation}\label{eq2.F}
(z;q)_\la:=\prod_{i=1}^{\ell(\la)}(zq^{1-i};q)_{\la_i}=\prod_{(i,j)\in\la}(1-zq^{j-i}).
\end{equation}

\subsection{Interpolation symmetric functions: statement of result} 

\begin{theorem}\label{thm2.A}
{\rm(i)} For each partition $\mu$ there exists a unique, within a scalar factor, symmetric function  $I_\mu(\,\cdot\,;q)\in\Sym$ with the following properties{\rm:}
\begin{gather*}
\deg I_\mu(\,\cdot\,;q)=|\mu|, \quad I_\mu(X(\mu);q)\ne0, \\
\quad I_\mu(X(\la);q)=0 \quad \text{\rm if  $|\la|\le|\mu|$ and $\la\ne\mu$}.
\end{gather*}

{\rm(ii)} Moreover, $I_\mu(X(\la))=0$ whenever $\la$ does not contain $\mu$.

{\rm(iii)} The expansion of the function $I_\mu(\,\cdot\,;q)$ on the Schur functions has the form
$$
I_\mu(\,\cdot\,;q)=\const\sum_{\nu\subseteq\mu}\si(\mu,\nu;q) S_\nu(\,\cdot\,),
$$
where
\begin{equation}\label{eq2.A}
\si(\mu,\nu;q):=(-1)^{|\mu|-|\nu|}q^{n(\mu)-n(\mu')-n(\nu)+n(\nu')}\det\left[\frac1{(q;q)_{\mu_i-\nu_k-i+k}}\right]_{i,k=1}^{\ell(\mu)}.
\end{equation}

{\rm(iv)} In what follows we use the normalization in which the constant in front of the sum in\/ {\rm(iii)} equals\/ $1$. Then
\begin{equation}\label{eq2.D}
I_\mu(X(\mu);q)=H(\mu;q):=q^{-\sum_{i=1}^{\ell(\mu)}\mu_i(\mu_i-i+1)} \prod_{(i,j)\in\mu}(1-q^{h(i,j)}),
\end{equation}
where $h(i,j)$ is the hook length corresponding to the box $(i,j)$.

\end{theorem}

The proof is given below in \S\ref{sect2.5}. 

In connection with formula \eqref{eq2.A} it is worth noting that, by definition, 
$$
\frac1{(q;q)_r}=0 \qquad \text{for $r=-1,-2,\dots$},
$$
which agrees with the formula 
$$
\frac1{(q;q)_r}=\frac{(q^{r+1};q)_\infty}{(q;q)_\infty}.
$$

It follows from (iii) that  
$$
I_\mu(\,\cdot\,;q)=S_\mu(\,\cdot\,)+\text{lower degree terms}.
$$

Furthermore, for $I_\mu(\ccdot;q)$ there is a combinatorial formula (Theorem \ref{thm2.B}). 

\subsection{Interpolation symmetric polynomials in $N$ variables}

For an arbitrary partition $\mu$ and arbitrary natural number $N\ge\ell(\mu)$ we set
\begin{equation}\label{eq2.G}
I_{\mu\mid N}(x_1,\dots,x_N;q):=\frac{\det\left[(x_j-q^{N-1})(x_j-q^{N-2})\dots(x_j-q^{-\mu_i+i})\right]_{i,j=1}^N}{\prod\limits_{1\le i<j\le N}(x_i-x_j)}
\end{equation}
(the number of factors in the numerator equals $\mu_i+N-i$). Evidently, $I_{\mu\mid N}(\ccdot;q)$ lies in the algebra $\Sym(N)$ of symmetric polynomials and the top homogeneous component of $I_{\mu\mid N}(\ccdot;q)$ coincides with the Schur polynomial $S_{\mu\mid N}$. 

The polynomial $I_{\mu\mid N}(\ccdot;q)$ is a particular case of multiparameter Schur polynomials  (see Macdonald \cite[6th variation]{M-SLC}) and also a particular case of the interpolation polynomials studied in the works of Knop \cite{Knop}, Okounkov \cite{Ok-MRL}, \cite{Ok-CM}, \cite{Ok-AAM}, \cite{Ok-FAA}, and Sahi \cite{Sahi1}, \cite{Sahi2} (see also Koornwinder's survey \cite{K-SLC}).

The connection with the notations of these papers is the following. Our polynomial $I_{\mu\mid N}(x_1,\dots,x_N;q)$ is equal to:

$\bullet$ $P_\mu(x_1,\dots,x_N)$ with parameters $(1/q,1/q)$, in the notation of Knop \cite{Knop};

$\bullet$ $P^*_\mu(x_1,x_2q^{-1},\dots,x_N q^{1-N}; 1/q,1/q)$, in the notation of Okounkov \cite{Ok-MRL}, \cite{Ok-CM};

$\bullet$ $q^{(N-1)|\mu|}R_\mu(x_1 q^{1-N}, \dots, x_N q^{1-N}; q,q)$, in the notation of Sahi \cite{Sahi2}.

\begin{proposition}\label{prop2.A}
The expansion of the polynomial $I_{\mu\mid N}(\,\cdot\,;q)$ on the Schur polynomials has the form
$$
I_{\mu\mid N}(\,\cdot\,;q)=\sum_{\nu\subseteq\mu}\frac{(q^N;q)_\mu}{(q^N;q)_\nu}\si(\mu,\nu;q) S_{\nu\mid N},
$$
where the coefficients $\si(\mu,\nu;q)$ do not depend on $N$ and are given by the formula \eqref{eq2.A}.
\end{proposition}

This expansion is readily obtained from \cite[(1.12) and (1.9)]{Ok-MRL} by setting $t=q$. Here is an independent elementary derivation. 

\begin{proof}
(a) Let us prove the formula
\begin{multline}\label{eq2.B}
(x-q^{N-1})(x-q^{N-2})\dots(x-q^{N-m})\\
=\sum_{n=0}^m(-1)^{m-n}q^{(Nm-\frac12 m^2-\frac12 m)-(Nn-\frac12 n^2-\frac12 n)} \frac{(q;q)_m}{(q;q)_n(q;q)_{m-n}} x^n.
\end{multline}

Expanding the left-hand side in \eqref{eq2.B} on the powers of the variable $x$ we obtain  
$$
\sum_{n=0}^m(-1)^{m-n}e_{m-n}(q^{N-1},\dots,q^{N-m}) x^n,
$$
where $e_r$ denotes the elementary symmetric polynomial of degree $r$. This can be rewritten as  
$$
\sum_{n=0}^m(-1)^{m-n}q^{(N-m)(m-n)}e_{m-n}(1,q,\dots,q^{m-1}) x^n.
$$
Using the well-known formula \cite[Ch. I, \S2, Example 3]{M}
$$
e_r(1,\dots,q^{m-1})=q^{r(r-1)/2}\frac{(q;q)_m}{(q;q)_r(q;q)_{m-r}},
$$
we obtain
$$
\sum_{n=0}^m(-1)^{m-n}q^{(N-m)(m-n)+\frac12(m-n)(m-n-1)}\frac{(q;q)_m}{(q;q)_n(q;q)_{m-n}} x^n,
$$
and the desired expression arises after reduction of similar terms in the exponent of $q$. As will be seen in what follows, it is important for us that after cancellation the term containing the product $mn$ disappears. 

(b) Denote by $\det[A(i,j)]$ the determinant in the right-hand side of the formula for $I_{\mu\mid N}(\ccdot;q)$. Introducing the notation
$$
m_i:=\mu_i+N-i,  \qquad 1\le i\le N,
$$
we rewrite the matrix elements $A(i,j)$ as
$$
A(i,j)=(x_j-q^{N-1})\dots(x_j-q^{N-m_i}).
$$
Next, applying \eqref{eq2.B}, we may write
$$
A(i,j)=\sum_{n=0}^\infty B(i,n)C(n,j),
$$
where
\begin{gather*}
B(i,n):=(-1)^{m_i-n}q^{(Nm_i-\frac12 m_i^2-\frac12 m_i)-(Nn-\frac12 n^2-\frac12 n)} \frac{(q;q)_{m_i}}{(q;q)_n(q;q)_{m_i-n}}, \\
C(n,j):=x_j^n.
\end{gather*}
It follows
\begin{multline}\label{eq2.C}
I_{\mu\mid N}(x_1,\dots,x_N;q)=\frac{\det[A(i,j)]_{i,j=1}^N}{\prod\limits_{1\le i<j\le N}(x_i-x_j)}\\
=\sum_{n_1>\dots>n_N\ge0}\det[B(i,n_k)]_{i,k=1}^N\frac{\det[C(n_k,j)]_{k,j=1}^N}{\prod\limits_{1\le i<j\le N}(x_i-x_j)}.
\end{multline}
The sum is actually finite, because for $n_1>m_1$ the first column of the matrix $[B(i,n_k)]$ vanishes due to the factor $(q;q)_{m_i-n_1}$ in the denominator. 

The determinant of this matrix is transformed in the following way:
\begin{multline*}
\det[B(i,n_k)]=(-1)^{\sum_i(m_i-n_i)}q^{\sum_i\{N(m_i-n_i)-\frac12[m_i^2+m_i-n_i^2-n_i]\}}\\
\times\prod_i\frac{(q;q)_{m_i}}{(q;q)_{n_i}}\det\left[\frac1{(q;q)_{m_i-n_k}}\right]_{i,k=1}^N.
\end{multline*}

Note that the latter determinant vanishes whenever $n_j>m_j$ for some $j$. Indeed, then  $m_i-n_k<0$ for $i\ge j\ge k$, and hence the matrix elements with indices  $(i,k)$ vanish for $i\ge j\ge k$, which entails the vanishing of the determinant. Thus, we may assume that $m_j\ge n_j$ for all $j=1,\dots,N$.

Now we pass from collections $(n_1>\dots>n_N\ge0)$ to partitions $\nu$ with  $\ell(\nu)\le N$ by setting
$$
n_i=\nu_i+N-i, \qquad 1\le i\le N.
$$
Then the inequalities $m_j\ge n_j$ mean that $\nu\subseteq\mu$. Now we rewrite the expressions for the determinants $\det[B(i,n_k)]$ and $\det[C(n_k,j)]$ in terms of $\nu$. 

For the first determinant we obtain
$$
\sum_{i=1}^N(m_i-n_i)=|\mu|-|\nu|,
$$
\begin{multline*}
\sum_{i=1}^N\left\{N(m_i-n_i)-\frac12[m_i^2+m_i-n_i^2-n_i]\right\}\\
=\sum_{i=1}^N\left\{-\frac12\mu_i(\mu_i-1)+\mu_i(i-1)+\frac12\nu_i(\nu_i-1)-\nu_i(i-1)\right\}\\
=n(\mu)-n(\mu')-n(\nu)+n(\nu'),
\end{multline*}
$$
\prod_{i=1}^N\frac{(q;q)_{m_i}}{(q;q)_{n_i}}=\frac{(q^N;q)_\mu}{(q^N;q)_\nu}.
$$
Finally,
$$
\det\left[\frac1{(q;q)_{m_i-n_k}}\right]_{i,k=1}^N=\det\left[\frac1{(q;q)_{\mu_i-\nu_k-i+k}}\right]_{i,k=1}^{\ell(\mu)}.
$$

The second determinant is rewritten easily:
$$
\det[C(n_k,j)]_{k,j=1}^N=\det[x_j^{\nu_k+N-k}]_{k,j=1}^N,
$$
and after division by $\prod_{i<j}(x_i-x_j)$  we obtain $S_{\nu\mid N}(x_1,\dots,x_N)$.

Substituting these expression in \eqref{eq2.C} we obtain the desired formula for $I_{\mu\mid N}(\ccdot;q)$.
\end{proof}

By analogy with the definition of the vectors $X(\la)\in\ell^1$ we introduce the following definition: for any $N\ge\ell(\la)$, 
$$
X_N(\la):=(q^{-\la_1}, q^{-\la_2+1}, q^{-\la_3+2},\dots,q^{-\la_N+N-1})\in\R^N.
$$

\begin{proposition}\label{prop2.B}
Let $\mu$ and $\la$ be partitions with $\ell(\ccdot)\le N$. 

{\rm(i)} If $\la$ does not contain $\mu$, then $I_{\mu\mid N}(X_N(\la);q)=0$.

{\rm(ii)} $I_{\mu\mid N}(X(\mu);q)=H(\mu;q)$, where $H(\mu;q)$ is defined in \eqref{eq2.D}.
\end{proposition}

\begin{proof}
(i) Consider the matrix $[A(i,j)]$ of size $N\times N$ with elements
\begin{align*}
A(i,j):&=(X(\la)_j-q^{N-1})\dots(X_N(\la)_j-q^{-\mu_i+i)})\\
&=(q^{-\la_j+j-1}-q^{N-1})\dots(q^{-\la_j+j-1}-q^{-\mu_i+i}).
\end{align*}
If $\la$ does not contain $\mu$, then there exists at least one index $k$ such that $\la_k<\mu_k$. Then for all pairs $(i,j)$, such that $i\le k\le j$, the inequality $\la_j<\mu_i$ holds. It follows that for every such pair $(i,j)$, one of the factors in the expression for $A(i,j)$ vanishes. This in turn implies that $\det[A(i,j)]=0$ and hence $I_{\mu\mid N}(X(\la);q)=0$.

(ii) A similar argument shows that in the case $\la=\mu$ the matrix $[A(i,j)]$ is lower triangular. Therefore, its determinant equals the product of the diagonal elements. It follows that 
$$
I_{\mu\mid N}(X_N(\mu);q)=\dfrac{\prod\limits_{i=1}^N(q^{-\mu_i+i-1}-q^{N-1})\dots(q^{-\mu_i+i-1}-q^{-\mu_i+i})}{\prod\limits_{1\le i<j\le N}(q^{-\mu_i+i-1}-q^{-\mu_j+j-1})}.
$$
Extracting from each factor the quantity $q^{-\mu_i+i-1}$ we obtain 
$$
q^{-\sum_i\mu_i(\mu_i-i+1)}\dfrac{\prod\limits_{i=1}^N(q;q)_{\mu_i+N-i}}{\prod\limits_{1\le i<j\le N}(1-q^{\mu_i-\mu_j-i+j})}=q^{-\sum_i\mu_i(\mu_i-i+1)}\prod_{(i,j)\in\mu}(1-q^{h(i,j)})=H(\mu;q),
$$
where the first equality follows from \cite[Ch. I, \S3, Example 1, (3)]{M} and the second equality is the definition \eqref{eq2.D}. 
\end{proof}

\subsection{Approximation $\Sym(N)\to\Sym$}

Denote by $\Sym_{\le d}$ the subspace in $\Sym$ formed by the elements of degree at most $d$, where $d=0,1,2,\dots$\,. In the similar way we define the subspace $\Sym_{\le d}(N)$. These subspaces have finite dimension and the canonical projection  $\Sym\to\Sym(N)$ determines a projection $\Sym_{\le d}\to\Sym_{\le d}(N)$. Under the condition $N\ge d$ the latter projection is a linear isomorphism and hence has the inverse:
$$
\iota_{d,N}: \Sym_{\le d}(N)\to \Sym_{\le d}, \qquad N\ge d.
$$
Evidently, $\iota_{d+1,N}$ extends $\iota_{d,N}$ for every $N\ge d+1$.

\begin{definition}\label{def2.A}
Let us say that a sequence $\{F_N\in\Sym(N): N\ge N_0\}$ \emph{converges} to a certain element $F\in\Sym$ if $\sup_N\deg F_N<\infty$ and for $d$ large enough
$$
\lim_{N\to\infty}\iota_{d,N}(F_N)\to F
$$
in the finite-dimensional space $\Sym_{\le d}$.  Then we write $F_N\to F$ or $\lim_N F_N=F$.
\end{definition}

Here are two simple propositions.

\begin{proposition}\label{prop2.C}
The condition of convergence $F_N\to F$ is equivalent to the following{\rm:}  $\sup_N \deg F_N<\infty$ and as  $N\to\infty$, the coefficients of the expansion of the elements $F_N$ on the Schur polynomials converge to the corresponding coefficients in the expansion of the element  $F$ on the Schur function.
\end{proposition}

\begin{proof} Under the canonical projection $\Sym\to\Sym(N)$, the Schur function $S_\nu$ turns into the Schur polynomial $S_{\nu\mid N}$ for every $\nu$ with $\ell(\nu)\le N$. Hence, conversely, 
$$
\iota_{d,N}(S_{\nu\mid N})=S_\nu \quad \text{при} \quad \ell(\nu)\le N, \quad |\nu|\le d.
$$
This makes the proposition evident.
\end{proof}

For an arbitrary $X=(x_1,x_2,\dots)\in\ell^1$ set 
$$
X_N:=(x_1,\dots,x_N)=(x_1,\dots,x_N,0,0,\dots).
$$ 

\begin{proposition}\label{prop2.D}
If $F_N\to F$, then $F_N(X_N)\to F(X)$ for any $X\in\ell^1$. 
\end{proposition}

\begin{proof}
By virtue of Proposition \ref{prop2.C} it suffices to check that $S_{\nu\mid N}(X_N)\to S_\nu(X)$ for any  $\nu$, but this holds true because $S_{\nu\mid N}(X_N)=S_\nu(X_N)$ and $X_N\to X$ in the metric of the space $\ell^1$.
\end{proof}

\subsection{Proof of Theorem \ref{thm2.A}}\label{sect2.5}
We apply Propositions  \ref{prop2.A} and \ref{prop2.B},  Definition \ref{def2.A}, and Propositions \ref{prop2.C} and  \ref{prop2.D}.

Fix a partition $\mu$ and set 
$$
I_\mu(\ccdot;q):=\sum_{\nu\subseteq\mu}\si(\mu,\nu;q) S_\nu\in\Sym,
$$
where the coefficients $\si(\mu,\nu;q)$ are defined in  \eqref{eq2.A}.

Note that $\lim_{N\to\infty}(q^N;q)_\la=1$ for any fixed partition $\la$. From this and Proposition  \ref{prop2.A} it follows that  
$$
I_{\mu\mid N}(\ccdot;q)\to I_\mu(\ccdot;q)
$$
in the sense of Definition \ref{def2.A}. 

Therefore, $I_{\mu\mid N}(X_N;q)\to I_\mu(X;q)$ for any $X\in\ell^1$. In particular, this holds for $X=X(\la)$ with an arbitrary partition $\la$. Note that if  $X=X(\la)$, then $X_N$ coincides with  $X_N(\la)$, which implies that $I_{\mu\mid N}(X_N(\la);q)\to I_\mu(X(\la);q)$. Applying Proposition  \ref{prop2.B} we see that $I_\mu(X(\la))=0$ whenever  $\la$ does not contain $\mu$,  and $I_{\mu}(X(\mu);q)=H(\mu;q)$. 

Thus, we have proved all the claims of Theorem \ref{thm2.A}, except that about uniqueness in (i). But it is a formal consequence of the existence claim, because $\dim\Sym_{\le d}$ equals the number of partitions $\la$ with $|\la|\le d$.

\subsection{Combinatorial representation}

The theorem below is a complement to Theorem \ref{thm2.A}. Recall the definition of a \emph{reverse tableau} of shape $\mu$ (see \cite[\S11]{OO-AA}): this is a filling of the boxes of the diagram $\mu$ by numbers which weakly decrease along the rows from left to right and strictly decrease down the columns. We denote by $\Tab(\mu)$ the set of all reverse tableaux of shape $\mu$ and with the values in $\{1,2,\dots\}$.

\begin{theorem}\label{thm2.B}
For any partition  $\mu$ and any vector $X=(x_1,x_2,\dots)\in\ell^1$ the following combinatorial formula holds{\rm:}
\begin{equation}\label{eq2.E}
I_\mu(X;q)=\sum_{T\in\Tab(\mu)}\prod_{(i,j)\in\mu}(x_{T(i,j)}-q^{T(i,j)+i-j-1}),
\end{equation}
where the series on the right absolutely converges, and the same holds with the parentheses removed. 
\end{theorem}

\begin{example}
For the three diagrams $(1)$, $(1^2)$, and $(2)$, the formula \eqref{eq2.E} gives:
\begin{align*}
I_{(1)}(X;q)&=\sum_{n=1}^\infty(x_n-q^{n-1})=S_{(1)}(X)-\frac1{1-q},\\
I_{(1^2)}(X;q)&=\sum_{n_1>n_2\ge1}(x_{n_1}-q^{n_1-1})(x_{n_2}-q^{n_2})\\
&=S_{(1^2)}(X)-\frac{q}{1-q}S_{(1)}(X)+\frac{q^2}{(1-q)(1-q^2)}\\
I_{(2)}(X;q)&=\sum_{n_1\ge n_2\ge1}(x_{n_1}-q^{n_1-1})(x_{n_2}-q^{n_2-2})\\
&=S_{(2)}(X)-\frac1{q(1-q)}S_{(1)}(X)+\frac1{q(1-q)(1-q^2)}.
\end{align*}
Note that the result agrees with claim (iii) of Theorem \ref{thm2.A}. 
\end{example}

\begin{proof}
(a) Let us check that the series in \eqref{eq2.E} absolutely converges and the same holds with the parentheses removed. Indeed, attach to every vector $X=(x_1,x_2,\dots)\in\ell^1$ the following vector  $\wt X=(\wt x_1,\wt x_2,\dots)$ with positive coordinates
$$
\wt x_n:=|x_n|+q^{n-\mu_1}, \qquad n=1,2,\dots\,.
$$
Evidently, $\wt X$ also lies in $\ell^1$.

On the other hand,
$$
|x_{T(i,j)}-q^{T(i,j)+i-j-1}|\le |x_{T(i,j)}|+q^{T(i,j)+i-j-1} \le |x_{T(i,j)}|+q^{T(i,j)-\mu_1}=\wt x_{T(i,j)}.
$$
Consequently our series is dominated by the convergent series  
$$
\sum_{T\in\Tab(\mu)}\prod_{(i,j)\in\mu}\wt x_{T(i,j)}=S_\mu(\wt X).
$$

(b) Observe that for $I_{\mu\mid N}(\ccdot;q)$ there is a combinatorial representation, which is similar to \eqref{eq2.E}; the only difference is that the summation is taken over the finite subset $\Tab(\mu,N)\subset \Tab(\mu)$ consisted of the reverse tableaux with the values in $\{1,\dots,N\}$:
$$
I_{\mu\mid N}(Y;q)=\sum_{T\in\Tab(\mu,N)}\prod_{(i,j)\in\mu}(y_{T(i,j)}-q^{T(i,j)+i-j-1}), \qquad Y=(y_1,\dots,y_n)\in\R^N.
$$
Indeed, this follows (after is a simple reformulation) from a more general result due to Okounkov \cite[Theorem III]{Ok-CM}.  

On the other hand, we know that 
$$
I_\mu(X;q)=\lim_{N\to\infty} I_{\mu\mid N}(X_N;q), \qquad \forall X\in\ell^1.
$$
Setting $Y=X_N$ and taking a formal limit transition as $N\to\infty$ in the combinatorial representation for $I_{\mu\mid N}(X_N,q)$ we obtain \eqref{eq2.E}. 

In order to justify this limit transition we note that the difference between the infinite series for  $I_\mu(X;q)$ and the finite series for $I_{\mu\mid N}(X_N;q)$ is a partial sum in \eqref{eq2.E} corresponding to the tableaux from $\Tab(\mu)\setminus\Tab(\mu, N)$. The estimate given above shows that the absolute value of this partial sum does not exceed
$$
\sum_{T\in\Tab(\mu)\setminus \Tab(\mu;q)}\prod_{(i,j)\in\mu}\wt x_{T(i,j)}=S_\mu(\wt X)-S_{\mu\mid N}(\wt X_N),
$$
which tends to zero as $N\to\infty$. 
\end{proof}

\subsection{Another approach}

Let us briefly describe a bit different way of constructing the symmetric functions $I_\mu(X;q)$, which gives the same result. For every $N=2,3,\dots$ we define a projection (an algebra morphism) $\Sym(N)\to\Sym(N-1)$ as the specialization of the $N$th variable at $q^{N-1}$. These projections are consistent with the filtration, and the filtered algebra $\varprojlim\Sym(N)$ is canonically isomorphic to $\Sym$ (here we substantially use the fact that $q<1$). Next, it is readily verified that for any fixed $\mu$, the elements $I_{\mu\mid N}(\ccdot;q)$ are consistent with the projections. This allows us to define $I_\mu(\ccdot;q)$ as the projective limit of these elements. 

\section{Symmetric functions $\Phi_\la(X;q,\alde)$}\label{sect3}

\subsection{A $q$-analogue of the beta distribution}\label{sect3.A}

The classical beta distribution is supported by the segment $[0,1]$. It can be carried over to the segment  $[-1,1]$ by a linear change of the variable, and in such form it serves as the weight measure for the Jacobi polynomials. 

Both variants of the beta distribution, on the segment $[0,1]$ and on the segment $[-1,1]$, have  $q$-analogues, which, however, substantially differ --- they are not reduced to each other by a change of the variable. Moreover, the corresponding systems of orthogonal polynomials (the little and big $q$-Jacobi polynomials) are on different levels of the $q$-hypergeometric hierarchy. We need the second, more complicated, $q$-analogue of the beta distribution, introduced in the note \cite{AA-PAMS} by Andrews and Askey. 

Recall that the parameter $q$ is fixed and belongs to $(0,1)$. We also introduce a quadruple of parameters $(a,b,c,d)$, where $a>0$, $b<0$, and the condition on $(c,d)$ will appear shortly. The distribution that we need is supported by a bounded countable subset of $\R$,
\begin{multline}\label{eq3.D}
\De_{q,a,b}= b^{-1} q^{\Z_{>0}}\cup a^{-1}q^{\Z_{>0}}\\
=\{b^{-1}q,b^{-1}q^2,b^{-1}q^3,\dots\}\cup\{\dots,a^{-1}q^3,  a^{-1}q^2,a^{-1}q\}\subset\R\setminus\{0\},
\end{multline}
and has the form
$$
\const W(x;a,b,c,d), \qquad  x\in\De_{q,a,b},
$$
where $\const$ is a normalizing constant factor and 
\begin{equation}\label{eq3.A}
W(x;q,a,b,c,d):=|x|\frac{(ax;q)_\infty (bx;q)_\infty}{(cx;q)_\infty (dx;q)_\infty}.
\end{equation}

The condition on $(c,d)$ is that the weight function $W(x;a,b,c,d)$ must be well defined and strictly positive at all points $x\in\De_{q,a,b}$. The numerator in  \eqref{eq3.A} is strictly positive on $\De_{q,a,b}$, so that it is necessary that the same be true for the denominator. The range of parameters $(c,d)$ satisfying this requirement consists of two parts:

$\bullet$ The \emph{principal series} consists of pairs of complex-conjugate numbers in   $\C\setminus\R$; in this case the numbers  $(cx;q)_\infty$ and  $(dx;q)_\infty$ are nonzero and complex-conjugate.

$\bullet$ The \emph{complementary series} consists of pairs of reals that lie together in one of the open intervals between two neighboring points of the two-sided infinite sequence
$$
\ldots<bq^{-3}<bq^{-2}<bq^{-1}<aq^{-1}<aq^{-2}<aq^{-3}<\ldots;
$$
in this case the numbers $(cx;q)_\infty$ and $(dx;q)_\infty$ are nonzero reals of the same sign.

The normalization factor is determined from 
$$
\const^{-1}=\sum_{x\in\De_{q,a,b}}W(x;a,b,c,d).
$$
As shown in \cite{AA-PAMS}, the series on the right can be summed explicitly and the result is given by a multiplicative formula.

\subsection{Big $q$-Jacobi polynomials}
We denote by $\varphi_\ell(x;q,a,b,c,d)$ the monic orthogonal polynomials on $\De_{q,a,b}$ with the weight function $W(x;q,a,b,c,d)$ given by \eqref{eq3.A}. The pair $(c,d)$ is assumed to be in  the principal or complementary series (see above).  Here is an explicit formula:
\begin{equation}\label{eq3.E}
\varphi_\ell(x; q, a,b,c,d):= \dfrac{\left(\dfrac{cq}a;q\right)_\ell \left(\dfrac{cq}b;q\right)_\ell}{c^\ell\left(\dfrac{cdq^{\ell+1}}{ab};q\right)_\ell}\,{}_3\phi_2 \left[\begin{matrix} q^{-\ell},\, \dfrac{cdq^{\ell+1}}{ab},\, cx\\ \dfrac{cq}a,\qquad\dfrac{cq}b\end{matrix} \,\Bigg|\,q;q\right],  \quad \ell=0,1,\dots\,. 
\end{equation}

The symmetry $c\leftrightarrow d$ is not evident from \eqref{eq3.E}, but it can be verified by making use of a transformation of the hypergeometric series. In the case $c=0$ the formula takes the indeterminate form $0/0$, but it can be resolved by the limit transition. 

The polynomials $\varphi_\ell(x;q,a,b,c,d)$ were called the \emph{big $q$-Jacobi polynomials} by Andrews and Askey, see their paper \cite[pp. 46--49]{AA}; see also  \cite[p. 442, Remarks]{KLS} (or \cite[end of \S3.5, Remarks]{KS}), \cite[\S14.5]{K-Additions} (the notation in these works is different from ours). 

Next, we define the $N$-variate analogues of the polynomials $\varphi_\ell(x;q,a,b,c,d)$ according to the scheme of \S\ref{sect1.A}:
$$
\varphi_{\la\mid N}(x_1,\dots,x_N;q,a,b,c,d):=\frac{\det[\varphi_{\la_i+N-i}(x_j)]_{i,j=1}^N}{\prod\limits_{1\le i<j\le N}(x_i-x_j)},
$$
where $\la$ ranges over the set of partitions with $\ell(\la)\le N$.

\subsection{Expansion on interpolation polynomials}
In what follows the pair of parameters  $(c,d)$ will vary together with $N$, so that we introduce a new notation by setting 
$$
(a,b,c,d)=(\al,\be,\ga q^{1-N}, \de q^{1-N}).
$$

\begin{definition}\label{def3.A}
We say that a quadruple of parameters $(\alde)$ is \emph{admissible} if $\al>0$, $\be<0$, and the pair $(\ga,\de)$ satisfies one of the following two conditions: 

$\bullet$ either $\ga=\bar\de\in\C\setminus\R$,

$\bullet$ or  $\ga$ and  $\de$ lie together in one of the open intervals between neighboring points of the two-sided infinite sequence  
$$
\ldots<bq^{-3}<bq^{-2}<bq^{-1}<0<aq^{-1}<aq^{-2}<aq^{-3}<\ldots.
$$
\end{definition}

In other words,  $(\ga,\de)$ belongs to the principal or complementary series, but in the second case it is additionally supposed that  $\ga$ and $\de$ are nonzero and have the same sign; then  multiplying by  $q^{1-N}$ does not take them out the complementary series. The exceptional case  $\ga=\de=0$ is excluded; it is examined separately in  \S\ref{sect5.2}.

In what follows we suppose that the quadruple of parameters  $(\alde)$ is admissible in the sense of Definition \ref{def3.A}.

\begin{proposition}\label{prop3.A}
The following expansion holds{\rm:}
\begin{multline*}
\varphi_{\la\mid N}(x_1,\dots,x_N;q,\al,\be,\ga q^{1-N},\de q^{1-N})\\
=\sum_{\mu\subseteq\la}\frac{(q^N;q)_\la}{(q^N;q)_\mu}\rho(\la,\mu;q,\alde) I_{\mu\mid N}(x_1\ga,\dots,x_N\ga;q),
\end{multline*}
where the coefficients $\rho(\la,\mu;q,\alde)$ do not depend on $N$ and are given by 
\begin{multline}\label{eq3.B}
\rho(\la,\mu;q, \alde)=(-1)^{|\la|-|\mu|}\ga^{-|\la|}\\
\times\det\left[\dfrac{\left(\dfrac{\ga}{\al}q^{\mu_k-k+2};q\right)_{\la_i-\mu_k-i+k}\left(\dfrac{\ga}{\be} q^{\mu_k-k+2};q\right)_{\la_i-\mu_k-i+k}}{q^{(\mu_k-k+1)(\la_i-\mu_k-i+k)}\left(\dfrac{\ga\de}{\al\be}q^{\la_i+\mu_k-i-k+3};q\right)_{\la_i-\mu_k-i+k}(q;q)_{\la_i-\mu_k-i+k}}\right]_{i,k=1}^{\ell(\la)}
\end{multline}
\end{proposition}

\begin{proof}
We argue as in the proof of Proposition \ref{prop2.A}. 
Set
$$
\ell_i:=\la_i+N-i, \quad m_i:=\mu_i+N-i, \qquad 1\le i\le N,
$$
and examine the expansion 
$$
P_\ell(x;q,\al,\be,\ga q^{1-N}, \de q^{1-N})=\sum_{m=0}^\infty \wt\rho(\ell,m) (x\ga-q^{N-1})\dots(x\ga -q^{N-m}).
$$
The coefficients $\wt\rho(\ell,m)$, of course, vanish for $m>\ell$, so that the series actually terminates. In terms of these coefficients, the desired expansion looks as follows:
\begin{multline}\label{eq3.C}
\varphi_{\la\mid N}(x_1,\dots,x_N;q,\al,\be,\ga q^{1-N},\de q^{1-N})\\=\sum_{\mu\subseteq\la}\ga^{N(N-1)/2} \det[\wt\rho(\ell_i,m_k)]_{i,k=1}^N I_{\mu\mid N}(x_1\ga,\dots,x_N\ga;q).
\end{multline}  
Here the factor $\ga^{N(N-1)/2}$ arises because of the transformation $x_i\mapsto x_i\ga$  in the denominator of the formula for $I_{\mu\mid N}(X;q)$, see \eqref{eq2.G}.  

The coefficients $\wt\rho(\ell,m)$ are computed directly from the formula \eqref{eq3.E}:
$$
\wt\rho(\ell,m)=(-1)^{\ell-m}\ga^{-\ell} \frac{(q;q)_\ell}{(q;q)_m} \dfrac{\left(\dfrac\ga\al q^{m-N+2}\right)_{\ell-m}\left(\dfrac\ga\be q^{m-N+2}\right)_{\ell-m}}{q^{(m-N+1)(\ell-m)}\left(\dfrac{\ga\de}{\al\be}q^{\ell+m-2N+3}\right)_{\ell-m}(q;q)_{\ell-m}}.
$$
Substituting them in \eqref{eq3.C} we obtain the desired result. As in the proof of Proposition  \ref{prop2.A}, the initial determinant of order $N$ is reduced to a determinant of order $\ell(\la)$ due to the factors $(q;q)_{\ell-m}$ in the denominator.
\end{proof}

Combining Propositions \ref{prop3.A} and \ref{prop2.A} we immediately obtain  

\begin{corollary}\label{cor3.A}
The following expansion holds{\rm:}
\begin{multline*}
\varphi_{\la\mid N}(x_1,\dots,x_N;q,\al,\be,\ga q^{1-N},\de q^{1-N})\\
=\sum_{\nu\subseteq\la}\frac{(q^N;q)_\la}{(q^N;q)_\nu}\left(\,\sum_{\mu:\, \nu\subseteq\mu\subseteq\la}\rho(\la,\mu; q,\alde)\si(\mu,\nu;q)\right) S_{\nu\mid N}(x_1,\dots,x_N),
\end{multline*}
where the coefficients $\rho(\la,\mu;q,\alde)$ and $\si(\mu,\nu;q)$ do not depend on $N$ and are given by  \eqref{eq3.B} and \eqref{eq2.A}.
\end{corollary}

\subsection{Construction of symmetric functions $\Phi_\la(X;q,\alde)$}

Recall that the quadruple of parameters $(\alde)$ is assumed to be admissible in the sense of Definition \ref{def3.A}.

\begin{theorem}\label{thm3.A}
{\rm(i)} For any partition $\la$ there exists a limit in the sense of Definition \ref{def2.A}
$$
\Phi_\la(\ccdot;q,\alde):=\lim_{N\to\infty}\varphi_{\la\mid N}(\ccdot;q,\al,\be,\ga q^{1-N}, \de q^{1-N})\in\Sym.
$$

{\rm(ii)} The symmetric functions $\Phi_\la(\ccdot;q,\alde)$ are expanded on the interpolation functions and on the Schur functions as follows{\rm:}
\begin{align*}
\Phi_\la(X;q,\al,\be,\ga,\de)&=\sum_{\mu\subseteq\la}\rho(\la,\mu;q,\alde)I_\mu(X\ga;q)\\
&=\sum_{\nu\subseteq\la}\left(\,\sum_{\mu:\, \nu\subseteq\mu\subseteq\la}\rho(\la,\mu; q,\alde)\si(\mu,\nu;q)\right) S_\nu(X),
\end{align*}
where the coefficients $\rho(\la,\mu;q,\alde)$ and $\si(\mu,\nu;q)$ are given by \eqref{eq3.B} and \eqref{eq2.A}.
\end{theorem}

In particular, it follows from (ii) that the top homogeneous component of the function $\Phi_\la(\ccdot;q,\alde)$ coincides with the Schur function $S_\la$, which in turn implies that the elements $\Phi_\la(\ccdot;q,\alde)$ form a basis in $\Sym$.

\begin{proof}
Observe that
$$
\lim_{N\to\infty}(q^N;q)_\la=\lim_{N\to\infty}(q^N;q)_\nu=1.
$$
Then it is seen from the formula of Corollary \ref{cor3.A} that the coefficients in the expansion of the polynomials $\varphi_{\la\mid N}(\ccdot;q,\al,\be,\ga q^{1-N}, \de q^{1-N})$ on the Schur polynomials converge as $N\to\infty$, which yields (i). The same formula also provides the limit values of the coefficients, which gives (ii). 

\end{proof}

\section{Infinite-dimensional beta distribution and orthogonality}\label{sect4}

\subsection{The space $\Om$ and the embedding $\Sym\hookrightarrow C(\Om)$}\label{sect4.A}

As before, we fix two parameters $\al>0$ and $\be<0$. In accordance with \eqref{eq3.D} we set\begin{multline*}
\De_{q,\al,\be}= \be^{-1} q^{\Z_{>0}}\cup \al^{-1}q^{\Z_{>0}}\\
=\{\be^{-1}q,\be^{-1}q^2,\be^{-1}q^3,\dots\}\cup\{\dots, \al^{-1}q^3,\al^{-1}q^2,\al^{-1}q\}\subset\R\setminus\{0\}.
\end{multline*}

By a \emph{configuration in $\De_{q,\al,\be}$} we mean an arbitrary subset of $\De_{q,\al,\be}$. The set of all configurations is denoted by $\Om$. Identifying $\Om$ with $\{0,1\}^{\De_{q,\al,\be}}$, we endow $\Om$ with the direct product topology. Then $\Om$ becomes a compact space homeomorphic to the Cantor set.  

The space $\Om$ is the disjoint union of the subsets $\Om_\fin$ and $\Om_\infty$ (finite and infinite configurations). The subset $\Om_\fin$ is countable and dense in $\Om$. Its complement $\Om_\infty$ is uncountable and also dense in $\Om$. Next,  
$$
\Om_\fin=\bigcup_{N=0}^\infty\Om_N, 
$$
where $\Om_N\subset\Om$ consists of $N$-point configurations. 

Denote by $C(\Om)$ the space of continuous real-valued functions on $\Om$ with the supremum norm. This is a Banach algebra over $\R$ with pointwise operations. Observe that the algebra $\Sym$ is realized, in a natural way, as a dense subalgebra in  $C(\Om)$. Namely, the value of a function $F\in\Sym$ at a given configuration $X\in\Om$ is obtained by substituting  the points $x\in X$, enumerated in an arbitrary order, as arguments of the symmetric function $F(x_1,x_2,\dots)$, and  if the configuration is finite, then one adds a countable set of $0$'s. The definition is correct, because the resulting sequence of arguments lies in $\ell^1$, and the result does not depend on the enumeration chosen because of the very definition of the algebra $\Sym$. The continuity of $F(X)$ on $\Om$ is readily checked. Finally, the fact that  $\Sym$ is dense in  $C(\Om)$ follows from the Stone--Weierstrass theorem, because the functions from $\Sym$ obviously separate points of the space $\Om$.  

From the definition \eqref{eq3.D} it is seen that the set $\De_{q,\al,\be}$ can be identified, in a natural way, with  $\Z_{>0}\sqcup\Z_{>0}$ for any admissible values of the triple $(q,\al,\be)$. Therefore, the space $\Om$ is essentially the same for all  $(q,\al,\be)$. However, its concrete realization and the embedding $\Sym\hookrightarrow C(\Om)$ depend on $(q,\al,\be)$.

\subsection{Orthogonality of functions $\Phi_\la(\ccdot;q,\alde)$}

Let us  list the main  definitions and facts, which are used below in the formulation of Theorem \ref{thm4.A}. 

$\bullet$ We are dealing with the algebra $\Sym$ of symmetric functions over $\R$. 

$\bullet$ In \S\ref{sect3}, we defined an inhomogeneous basis in $\Sym$ whose elements are denoted by $\Phi_\la(\ccdot;q,\alde)$. Here the index $\la$ ranges over the set of partitions and  the top homogeneous component of the element $\Phi_\la(\ccdot;q,\alde)$ is the Schur function $S_\la$.

$\bullet$ The assumptions on the parameters are the following:  $0<q<1$ and the quadruple $(\alde)$ is admissible in the sense of Definition \ref{def3.A}. 

$\bullet$ Next, in \S\ref{sect4.A}, we defined a totally disconnected compact space $\Om$ (its concrete realization depends on the parameters $(q,\al, \be)$) and an embedding $\Sym\hookrightarrow C(\Om)$ of the algebra $\Sym$ into the algebra of continuous functions on $\Om$; the image of $\Sym$ in $C(\Om)$ is dense. 

$\bullet$ This embedding allows us to realize elements $\Phi_\la(\ccdot;q,\alde)\in\Sym$ as continuous functions $\Phi_\la(X;q,\alde)$ on the space $\Om$.

\begin{theorem}\label{thm4.A}
Let $q,\alde$ be fixed and $\la$ range over the set of partitions. 

{\rm(i)} There exists a unique probability Borel measure $M^{q,\alde}$ on the space $\Om$, such that the functions $\Phi_\la(X;q,\alde)$ form an orthogonal basis in the Hilbert space  $L^2(\Om,M^{q,\alde})$. 

{\rm(ii)} The norms of these functions are given by 
\begin{equation}\label{eq4.D}
\begin{gathered}
\Vert \Phi_\la(\ccdot;q,\alde)\Vert^2=q^{\sum_i \la_i(\la_i+3-2i)}(-s)^{|\la|}(-\al\be)^{-|\la|}\\
\times\dfrac{\left(\dfrac{\ga q}\al;q\right)_\la \left(\dfrac{\ga q}\be;q\right)_\la \left(\dfrac{\de q}\al;q\right)_\la \left(\dfrac{\de q}\be;q\right)_\la}{(sq;q)_{\wh\la}},
\end{gathered}
\end{equation}
where the symbol $(\ccdot)_\la$ is defined in \eqref{eq2.F}, 
$$
s:=\dfrac{\ga\de q}{\al\be},  \qquad \wh\la:=(2\la_1,2\la_1,2\la_2,2\la_2,\dots).
$$ 
 \end{theorem}

\noindent\textbf{Comments.} 1. The Young diagram $\wh\la$ is obtained from the diagram $\la$ by replacing each its box by a $2\times2$ square. 

2. From our assumptions on the parameters it follows that $\al\be<0$, $s<0$, and, moreover, the expression on the right-hand  side of \eqref{eq4.D} is strictly positive for every partition $\la$. 

3. As it will be seen from the proof, it is appropriate to regard the measure $M^{q,\alde}$ as an infinite-dimensional analogue of the $q$-beta distribution on $\De_{q,\al,\be}$ defined in \S\ref{sect3.A}. 

\begin{proof}

(a) From the general construction of \S\ref{sect1.A} it follows that the $N$-variate polynomials  $\varphi_{\la\mid N}(\ccdot;q,\al,\be, \ga q^{1-N}, \be q^{1-N})$ are orthogonal on $\Om_N$ with respect to the probability measure 
$M_N$ that assigns to a given configuration $(x_1,\dots,x_N)\in\Om_N$ the weight
\begin{equation}\label{eq4.H}
M_N(x_1,\dots,x_N):=\const\cdot \prod_{i=1}^N W(x_i;q,\al,\be,\ga q^{1-N}, \de q^{1-N})\cdot \prod_{1\le i<j\le N}(x_i-x_j)^2,
\end{equation}
where $\const$ is a normalization constant. Actually, we do not need to know the exact form of this measure, its existence already suffices. 

(b) Denote by $\wt M_N$ the same measure but viewed as a measure on the larger space $\Om\supset\Om_N$. Let us show that the measures $\wt M_N$ converge to a probability measure  $M$ on $\Om$ as $N\to\infty$. 

Indeed, let us denote the pairing between measures and functions by the angular brackets. Below we identify polynomials from $\Sym(N)$ with the corresponding functions on $\Om_N$; likewise, elements of the algebra $\Sym$ are identified with the corresponding functions on $\Om$.

The measure $M_N$ is characterized by the property 
$$
\langle M_N, \, \varphi_{\la\mid N}\rangle=\begin{cases} 1, & \la=\varnothing,\\ 0, & \la\ne\varnothing,\end{cases}
$$
where $\varnothing$ denotes the empty Young diagram  (= the zero partition) and  $\varphi_{\la\mid N}$ is a shorthand notation:
$$
\varphi_{\la\mid N}:=\varphi_{\la\mid N}(\ccdot;q,\al,\be,\ga q^{1-N},\de q^{1-N}).
$$

From Corollary \ref{cor3.A} it is seen that the transition matrix between the bases $\{\varphi_{\la\mid N}\}$ and $\{S_{\nu\mid N}\}$ is unitriangular and its elements converge as $N\to\infty$. It follows that for every $\nu$ there exists a limit $\lim_{N\to\infty}\langle M_N,\, S_{\nu\mid N}\rangle$.

On the other hand, $\langle M_N,\, S_{\nu\mid N}\rangle=\langle \wt M_N,\, S_\nu\rangle$. Consequently, for every $\nu$ there exists a limit for $\langle \wt M_N,\, S_\nu\rangle$. Because the linear span of the Schur functions is dense in  $C(\Om)$, we conclude that the measures $\wt M_N$ weakly converge to a probability Borel measure $M$. 

(c) From the above argument it follows that the limit measure satisfies the relations
$$
\langle M,\, S_\nu\rangle=\lim_{N\to\infty} \langle M_N,\, S_{\nu\mid N}\rangle \quad \text{for any partition $\nu$}.
$$
Due to stability of the Schur polynomials, the structure constants of the algebra  $\Sym(N)$ in the basis $\{S_{\nu\mid N}\}$ stabilize as $N$ grows, and their stable values coincide with the structure constants of the algebra $\Sym$ in the basis $\{S_\nu\}$. Therefore, more general limit relations holds:
$$
\langle M,\, S_\nu S_{\wt\nu}\rangle=\lim_{N\to\infty} \langle M_N,\, S_{\nu\mid N}S_{\wt\nu\mid N}\rangle \quad \text{for any pair $\nu, \wt\nu$ of partitions}.
$$

Now we use again the fact that the transition matrix between the bases $\{\varphi_{\la\mid N}\}$ and $\{S_{\nu\mid N}\}$ is unitriangular and its elements converge. Moreover, we know (Theorem \ref{thm3.A}) that their limit values coincide with the elements of the transition matrix between the bases  $\{\Phi_\la\}$ and $\{S_\nu\}$ of the algebra $\Sym$, where $\Phi_\la$ is a shorthand notation for $\Phi_\la(\ccdot;q,\alde)$. This gives us limit relations
$$
\langle M,\, \Phi_\la \Phi_{\wt\la}\rangle=\lim_{N\to\infty} \langle M_N,\, \varphi_{\la\mid N}\varphi_{\wt\la\mid N}\rangle \quad \text{for any pair $\la, \wt\la$ of partitions}.
$$
They imply that the images of the functions $\Phi_\la\in\Sym$ in the Hilbert space  $L^2(\Om,M)$ are pairwise orthogonal, and for their norms the following limit relations hold
\begin{equation}\label{eq4.A}
\Vert \Phi_\la\Vert^2_{L^2(\Om,M)}=\lim_{N\to\infty}\Vert \varphi_{\la\mid N}\Vert^2_{\ell^2(\Om_N,M_N)}.
\end{equation}

(d) However, the above argument does not allow yet to exclude the situation when the limit measure $M$ will be degenerate in the sense that some of the functions  $\Phi_\la$ will be zero almost everywhere with respect to the measure $M$. Let us show that this is impossible, by computing explicitly the norms by means of the formula \eqref{eq4.A}.

From the definition of the polynomials $\varphi_{\la\mid N}$ it follows that 
\begin{equation}\label{eq4.C}
\Vert \varphi_{\la\mid N}\Vert^2_{\ell^2(\Om_N,M_N)}=\prod_{i=1}^N\frac{h_{\la_i+N-i}(N)}{h_{N-i}(N)},
\end{equation}
where
$$
h_\ell(N):=\Vert \varphi_\ell(\ccdot;q,\al,\be,\ga q^{1-N}, \de q^{1-N})\Vert^2, \qquad \ell=0,1,2,\dots\,.
$$
For the latter quantity there is the following explicit expression:
\begin{align}
h_{\ell}(N)&=q^{2\ell}(-\al\be)^{-\ell} \label{eq4.B1}\\
&\times (\ga q^{2-N}/\al;q)_\ell (\ga q^{2-N}/\be;q)_\ell (\de q^{2-N}/\al;q)_\ell (\de q^{2-N}/\be;q)_\ell \label{eq4.B2}\\
&\times\dfrac{1}{(s q^{2-2N})_{2\ell}(s q^{3-2N})_{2\ell}} \label{eq4.B3}\\
&\times q^{\ell(\ell-1)/2}(s q^{2-2N};q)_\ell. \label{eq4.B4}
\end{align}
It is obtained from formulas given in \cite[\S14.5]{K-Additions} whose derivation can in turn be extracted from computations in \cite{AA} or \cite[\S7.3]{GR}.

Now we compute the product in the right-hand side of  \eqref{eq4.C} in consecutive order: 

$\bullet$ The contribution from  \eqref{eq4.B1} equals $q^{2|\la|}(-\al\be)^{-|\la|}$.

$\bullet$ The contribution from  \eqref{eq4.B2} equals $(\ga q/\al;q)_\la (\ga q/\be;q)_\la (\de q/\al;q)_\la (\de q/\be;q)_\la$.

$\bullet$ The contribution from  \eqref{eq4.B3} equals $1/(sq;q)_{\wh\la}$.

Note that all these expressions do not depend on $N$. There remains the contribution from\eqref{eq4.B4}; it already depends on $N$, and we find its asymptotics:

$\bullet$ The contribution from  \eqref{eq4.B4} equals
$$
\prod_{i=1}^{\ell(\la)}q^{\frac12\left\{(\la_i+N-i)(\la_i+N-i-1)-(N-i)(N-i-1)\right\}}\cdot \prod_{i=1}^{\ell(\la)} (sq^{2-N-i};q)_{\la_i}.
$$
Since $q^{-N}$ grows as  $N\to\infty$ while the number of factors does not depend on $N$, the second product can be written as 
$$
(-s)^{|\la|}\prod_{i=1}^{\ell(\la)}q^{\frac12\la_i(\la_i-2N-2i+3)}\cdot (1+O(q^N)).
$$
After multiplication by the first product the parameter $N$ is happily cancelled, and we obtain that the contribution from \eqref{eq4.B4} is
$$
(-s)^{|\la|}q^{\la_i^2+\la_i(-2i+1)}(1+O(q^N)).
$$

Finally, combining all the contributions together we obtain the expression from the right-hand of  \eqref{eq4.D}, together with the additional factor $(1+O(q^N))$, which tends to $1$. This completes the proof of the theorem. 

\end{proof}

\section{Concluding remarks}

\subsection{Complement to Theorem \ref{thm4.A}}

As above we assume that $(\alde)$ is an admissible quadruple of parameters (Definition \ref{def3.A}). According to claim (i) of Theorem \ref{thm4.A}, there exists a probability Borel measure $M^{q,\alde}$ on $\Om$, which makes the functions $\Phi_\la(X;q,\alde)$ orthogonal. Recall that the compact space $\Om$ consists of arbitrary configurations on $\De_{q,\al,\be}$, finite or infinite (see \S\ref{sect4.A}).

\begin{theorem}\label{thm5.A}
The measure $M^{q,\alde}$ is concentrated on the subset $\Om_\infty\subset\Om$ of infinite configurations. 
\end{theorem}

The proof is the same as that of Corollary 4.10 in \cite{GO-JFA}. The key fact is that the measures $M_N$ on $\Om_N$ defined by \eqref{eq4.H} form a \emph{coherent system} in the sense of \cite{GO-JFA}, which in turn is proved exactly as Theorem 4.7 in \cite{GO-JFA}.

\subsection{Exceptional case $\ga=\de=0$}\label{sect5.2}

Let us briefly describe what happens in this case (recall that it was excluded from consideration up to the present moment). 

If $\ga=\de=0$, then the symmetric functions $\Phi_\la(X;q,\alde)$ still exist and the measures $M_N$ from \eqref{eq4.H} still have a limit. On the other hand,  the right-hand side of \eqref{eq4.D} vanishes for any nonzero partition $\la$ because of the vanishing factor $(-s)^{\la}$ (note that $s=0$). This suggests that the limit measure $M=\lim M_N$ is a delta measure at a distinguished configuration $X\in\Om$. The next theorem shows that this holds true.

\begin{theorem}\label{thm5.B}
In the exceptional case $\ga=\de=0$, the limit measure $M=\lim M_N$ is the delta measure at a distinguished infinite configuration $X\in\Om$ --- the dense packing of the whole lattice $\De_{q,\al,\be}$.
\end{theorem}

This means in particular that  in the exceptional case, $\{M_N\}$ is an \emph{extreme} coherent system. The result is not interesting from the viewpoint of harmonic analysis. On the other hand, a similar situation holds for the sequence of Plancherel measures on the sets $\Y_N$ (Young diagrams with $N$ boxes). Based on this formal analogy one might ask if the measures $M_N$ with $\ga=\de=0$ possess (like the Plancherel measures on $\Y_N$) any interesting properties. 

\subsection{Limit transition as $q\to1$}

It is well known that the  big $q$-Jacobi polynomials $\varphi_\ell(x;q,a,b,c,d)$ (see \eqref{eq3.E}) can be degenerated to the classical Jacobi polynomials with parameters $A>-1$, $B>-1$ if one takes  a suitable limit regime with $q\to 1^-$. 

Namely, one has to set $c=aq^A$, $d=bq^B$, and then let $q\to 1^-$, $a\to1$, $b\to -1$.  As the parameters vary in this way, the weight measure of the big $q$-Jacobi polynomials converges to the probability measure 
$$ 
\frac{\Ga(A+B+2)}{2^{A+B+1}\Ga(A+1)\Ga(B+1)}(1-x)^A(1+x)^B dx, \quad -1<x<1,
$$
which is the weight measure of the classical Jacobi polynomials, and hence the polynomials also converge.

This shows that the big $q$-Jacobi polynomials can be viewed as a $q$-analogue of the classical Jacobi polynomials.

However, such a $q\to 1^-$ limit regime is incompatible with our large-$N$ limit transition. Indeed, in the $q\to1^-$ limit, the parameters $c$ and $d$ must be real and of opposite sign. On the other hand,  in our large $N$-limit, we use the polynomials with parameters $(a,b,c,d)=(\al,\be, \ga q^{1-N}, \de q^{1-N})$, where $(\alde)$ is fixed. If $\ga$ and $\de$ are real and of opposite sign, then for $N$ large enough the quadruple $(a,b,c,d)$ is outside the admissible range.  

This is a manifestation of the fact that our construction is a specific ``quantum'' effect, which is destroyed in the $q\to 1^-$ limit.


\begin{thebibliography}{AA}

\bibitem{AA-PAMS}
G. E. Andrews and R. Askey, $q$-Extension of the Beta Function. Proc. Amer. Math. Soc. 81 (1981), No. 1, 97--100.

\bibitem{AA}
G. E. Andrews and R. Askey, Classical orthogonal polynomials. In: Polyn\^omes orthogonaux. Lectures Notes in Math. 1171. Springer, 1985, pp. 36--62. 

\bibitem{GR}
G. Gasper and M. Rahman, Basic hypergeometric series. Second edition. Encyclopedia of Mathematics and its Applications 96. Cambridge Univ. Press, 2004. 

\bibitem{GO-JFA}
V. Gorin and G. Olshanski, A quantization of the harmonic analysis on the infinite-dimensional
unitary group. J.  Funct.  Anal. 270 (2016), 375--418.

\bibitem{Knop}
F. Knop, Symmetric and non-symmetric quantum Capelli polynomials. Comment. Math. Helv. 72 (1997), 84--100.

\bibitem{KLS}
R. Koekoek, P. A. Lesky and R. F. Swarttouw, Hypergeometric orthogonal polynomials
and their $q$-analogues, Springer-Verlag, 2010. 

\bibitem{KS}
R. Koekoek and R. F. Swarttouw, The Askey-scheme of hypergeometric orthogonal polynomials
and its $q$-analogue. Report 98-17, Faculty of Technical Mathematics and Informatics, Delft University of Technology, 1998; \url{http://aw.twi.tudelft.nl/~koekoek/askey/}; arXiv:math/9602214. 

\bibitem{K-Additions}
T.H. Koornwinder, Additions to the formula lists in ``Hypergeometric orthogonal polynomials and their $q$-analogue''€ by Koekoek, Lesky and Swarttouw, arXiv:1401.0815.

\bibitem{K-SLC}
T.H. Koornwinder, Okounkov's BC-type interpolation Macdonald polynomials and their $q$=1 limit.  S\'em. Lothar. Combin. B72a (2015), 27 pp.; arXiv:1408.5993. 

\bibitem{M-SLC}
I. G. Macdonald, Schur functions: theme and variations. S\'eminaire Lotharingien de Combinatoire 28 (1992), paper B28a, 35 pp. 

\bibitem{M}
I. G. Macdonald, Symmetric functions and Hall polynomials, 2nd ed. Oxford Univ. Press, 1995.

\bibitem{Ok-MRL}
A. Okounkov, Binomial formula for Macdonald polynomials and applications. Math. Research Lett. 4 (1997), 533--553.  

\bibitem{Ok-CM}
A. Okounkov, (Shifted) Macdonald polynomials: q-Integral representation and combinatorial formula. Comp. Math. 112 (1998), 147--182.

\bibitem{Ok-AAM}
A. Okounkov,  On Newton interpolation of symmetric functions: A characterization of interpolation Macdonald polynomials. Adv. Appl. Math. 20 (1998), 395--428.

\bibitem{Ok-FAA}
A. Yu. Okounkov, A remark on the Fourier pairing and the binomial formula for the Macdonald polynomials. Funct. Anal. Appl. 36 (2002), no. 2, 134--139. 

\bibitem{OO-AA}
A. Okounkov and G. Olshanski, Shifted Schur functions. Algebra i Analiz 9
(1997), no. 2, 73--146 (Russian); English version: St. Petersburg Mathematical
J., 9 (1998), 239--300.

\bibitem{Ols-JMS}
G. Olshanski, Laguerre and Meixner symmetric functions, and infinite-dimensional
diffusion processes. Zapiski Nauchnyh Seminarov POMI 378 (2010), 81-110;
reproduced in Journal of Mathematical Sciences (New York) 174 (2011), no.
1, 41--57; arXiv:1009.2037.

\bibitem{Ols-IMRN}
G. Olshanski, Laguerre and Meixner orthogonal bases in the algebra of symmetric
functions. Intern. Math. Research Notices 2012, 3615--3679.

\bibitem{P}
L. Petrov, $\mathfrak{sl}(2)$ Operators and Markov processes on branching graphs. J. Alg. Comb. 38 (2013), no. 3, 663--720. 

\bibitem{Rains}
E. M. Rains, $BC_n$ symmetric polynomials. Transf. Groups 10 (2005), No. 1, 63--132.

\bibitem{Sahi1}
S. Sahi, The spectrum of certain invariant differential operators associated to Hermitian
symmetric spaces. In: Lie theory and geometry, J.-L. Brylinski et al. (Eds.), Progress
Math. 123 Birkh\"auser, Boston 1994, 569--576.

\bibitem{Sahi2}
S. Sahi, Interpolation, integrality, and a generalization of Macdonald's polynomials. Intern. Math. Research Notices 1996, no. 10, 457--471.


\end{thebibliography}
\end{document}